\documentclass{amsart}
\usepackage{hyperref}
\usepackage{amsmath,amsthm,amssymb}

\theoremstyle{plain}
\newtheorem{theorem}{Theorem}

\newtheorem{hypothesis}{Hypothesis}

\theoremstyle{remark}

\theoremstyle{definition}

\newcommand{\p}{p}

\begin{document}

\title[On the connection problem for the $\protect \p$-Laplacian system]{On the connection 
problem for the $p$-Laplacian system for potentials with several global minima}
\author{Nikolaos Karantzas}
\address{Department of Mathematics\\ University of Athens\\ Panepistemiopolis\\ 15784 Athens\\ Greece}
\email{\href{mailto:nk1986@yahoo.co.uk}{\texttt{nk1986@yahoo.co.uk}}}
\thanks{The author was partially supported through the project PDEGE -- Partial Differential Equations Motivated by Geometric Evolution, co-financed by the European Union -- European Social Fund (ESF) and national resources, in the framework of the program Aristeia of the `Operational Program Education and Lifelong Learning' of the National Strategic Reference Framework (NSRF)}
\date{}

\begin{abstract}
We study the existence of solutions to systems of ordinary differential equations that involve the $p$-Laplacian for potentials with several global minima. We consider the connection problem for potentials with two minima in arbitrary dimensions and with three or more minima on the plane.
\end{abstract}

\maketitle

\section{Introduction}
We consider the problem of existence of solutions to systems of ordinary differential equations 
that involve the $p$-Laplacian operator, that is, systems of the form 
\[ ( |u_{x}|^{p-2} u_{x})_{x} - \frac{1}{q} \nabla W(u) = 0 \] 
for vector-valued functions $u$ and potentials $W$ that possess several global minima. 
The corresponding problem was considered in the papers Alikakos and Fusco \cite{AF} and Alikakos, Betel\'u, 
and Chen \cite{ABC} for the standard Laplacian and here we provide extensions of these results for $p>1$.

What is presented in the following is divided in two parts. 
In Section 2 we consider the problem in $\mathbb{R}^N$ for potentials 
with two global minima and state without proof an existence theorem 
together with a variational characterization of the connecting solutions 
(for the full proofs we refer to \cite{nikos-thesis}). 
The problem for potentials possessing three or more global minima (even for the case $p=2$)
is significantly harder and essentially open. In Section 3, by restricting ourselves
to $N=2$, we are able to exhibit a class of potentials for which we have reasonably 
complete results. In particular, we establish a uniqueness theorem and also give some 
examples of potentials that exhibit non-existence and non-uniqueness properties.

\section{The connection problem for potentials possessing two global minima}
Let $\Omega$ be an open and connected subset of 
$\mathbb{R}^{N}$ and $W: \Omega \to \mathbb{R}$ be a $C^{2}$ nonnegative 
potential function with two minima, that is, $W>0$ in 
$\mathbb{R}^{N} \setminus A$, with $W=0$ on 
$A=\{a^{+},a^{-}\}$. In \cite{AF}, Alikakos and Fusco analyze the 
existence of solutions to the Hamiltonian system
\begin{equation}
\label{one}
u_{xx}-\frac{1}{2} \nabla W(u) = 0, \text{ with } \lim_{x \to \pm \infty} u(x) = a^{\pm},
\end{equation}
where $u: \mathbb{R} \to \mathbb{R}^{N}$ is a vector-valued function. 
Such solutions are called {\em heteroclinic connections}. The system \eqref{one} 
represents the motion of $N$ material points of equal mass under the potential 
$-W(u)$, with $x$ standing for time and $u$ for position. The approach in \cite{AF} 
is variational and is based on Hamilton's principle of least action, that is, 
on the minimization of the action functional 
$A: W^{1,2}(\mathbb{R}, \mathbb{R}^{N}) \to \mathbb{R}$, defined as
\[ A(u)=\frac{1}{2} \int_{\mathbb{R}} (|u_{x}|^{2} + W(u))\, dx. \]
The method depends on the introduction of a constraint leading to the existence 
of local minimizers. The constraint can later be removed and therefore 
provide a solution to \eqref{one}.

In this section, we state without proof an extension of the results in \cite{AF} 
to the $p$-Laplacian operator. To this end, we consider the system
\begin{equation}
\label{two}
( |u_{x}|^{p-2} u_{x})_{x} - \frac{1}{q} \nabla W(u) = 0, \text{ with } \lim_{x \to \pm \infty} u(x)=a^{\pm}
\end{equation}
where $u: \mathbb{R} \to \mathbb{R}^{N}$ is again a vector-valued function and 
$p, q > 1$ are H\"{o}lder conjugates, that is, $1/p + 1/q = 1$. Alternatively, 
by Hamilton's principle of least action, the motion from one minimum of the potential 
to another is a critical point of the action functional 
$A_{p} : W^{1,p} ( [ t_{1} , t_{2} ], \mathbb{R}^{N} ) \to \mathbb{R}$, defined as
\begin{equation*}
A_{p} (u, (t_{1},t_{2})):= \int_{t_{1}}^{t_{2}} \left( \frac{|u_{x}|^{p}}{p}+\frac{W(u)}{q} \right) dx.
\end{equation*}
Therefore, the system \eqref{two} is the associated Euler--Lagrange equation with $t_{1}=-\infty$ and $t_{2}=+\infty$. 
To state our results, we assume that the following hypothesis holds.
 
\begin{hypothesis}
\label{hone}
The potential $W$ is such that $ \liminf_{|u| \to \infty} W(u) >0 $ and also 
there exists $R > 0$ such that the map $r \mapsto W(a^{\pm} + r\xi)$ has a strictly 
positive derivative for every $r \in (0 , R)$ and for every 
$\xi \in S^{N-1} := \{u \in \mathbb{R}^{N}: |u|=1 \}$, with $R < |a^{+} - a^{-}|$.
\end{hypothesis}

\noindent Then, under the above hypothesis, we have the following theorems.

\begin{theorem}
\label{thmone}
Let $W: \mathbb{R}^{N} \to \mathbb{R}$ be a non-negative $C^{2}$ potential function 
and let $a^{-} \neq a^{+} \in \mathbb{R}^{N}$ be such that $W(a^{\pm}) = 0$. 
Also assume that Hypothesis \ref{hone} holds. Then, there exists a 
connection $U$ between $a^{-}$ and $a^{+}$.
\end{theorem}

\begin{theorem}
\label{thmtwo}
Let $U$ be the minimizer provided by Theorem \ref{thmone} above and let $R$ 
be as defined in Hypothesis \ref{hone}. Also let $\mathcal{A}$ be the set that 
consists of all functions $u \in W_{\text{loc}}^{1,p} (\mathbb{R}; \mathbb{R}^{N})$ 
for which there exist $x_{u}^{-} < x_{u}^{+}$ (depending on $u$) such that 
\begin{equation*}
\begin{cases}
| u (x) - a^{-} | \leq R / 2, &\text{for all } x \leq x_{u}^{-},\\
| u (x) - a^{+} | \leq R / 2, &\text{for all } x \geq x_{u}^{+}.
\end{cases}
\end{equation*}
Then,
\[ A_{p} (U) = \min_{u \in \mathcal{A}} A_{p} (u). \]
\end{theorem}

The above theorems yield the existence of the desired heteroclinic connection and 
also a variational characterization. Our approach (for full details, 
see \cite{nikos-thesis}) follows the lines of the method 
established in \cite{AF} and therefore it is based on the minimization of the action 
functional $A_{p}$ over the whole real line. The proof of existence of a connection 
is generally straightforward, since the test functions constructed in Lemmas 3.1 
and 3.2 in \cite{AF} succeed in the pointwise reduction 
of both the gradient and the potential terms.

\section{The connection problem on the complex plane for potentials possessing several global minima}

In this section we extend to the $p$-Laplacian operator the existence and uniqueness results 
in Alikakos, Betel\'u, and Chen \cite{ABC}. Here, we consider potentials possessing 
three or more global minima and restrict ourselves to the planar case $N=2$ 
for which we identify $\mathbb{R}^{2}$ with the complex plane $\mathbb{C}$.

We tackle the problem by utilizing Jacobi's principle, which deals with 
curves and detects geodesics. Specifically, one considers the length functional
\begin{equation*}
L_{p}(u)=\int_{t_{1}}^{t_{2}} \sqrt[q]{W(u)} |u_{x}|\, dx,
\end{equation*}
which is independent of parametrizations and hence is more properly denoted by
\[ L_{p}(\Gamma)=\int_{\Gamma}\sqrt[q]{W(\Gamma)}\, d\Gamma . \]
Notice that the functional $A_{p}$ is defined on functions while the 
functional $L_{p}$ is defined on curves. The relationship between the two 
is that critical points of $L_{p}$ parametrized under the 
equipartition parametrization (that is, a parameter $t$ is such that 
$|u_{t}|^{p} = W (u)$, for all $t$ in an interval $(a, b)$) render critical 
points of $A_{p}$. In this case, we study the system of ordinary differential equations 
\[ (|u_{x}|^{p-2} u_{x}^{i})_{x} - \frac{1}{q} \frac{\partial W(u^{1}, u^{2})}{\partial u^{i}} = 0, \text{ for } i=1,2, \]
where $u = (u^{1}, u^{2}): \mathbb{R} \to \mathbb{R}^{2}$. We identify 
$u = (u^{1}, u^{2})$ with the complex number $z = u^{1} + i u^{2}$, 
and similarly we write $W(u)$ as $W(z)$. Since $W(\cdot)$ is non-negative,
we can write $W(z) = |f(z)|^{q}$ for some analytic function $f$. We can 
also verify that the equations in \eqref{two} are equivalent to 
\[ (|z_{x}|^{p-2} z_{x})_{x} = (f \overline{f})^{\frac{q-2}{2}} f \overline{f'}, \]
where the bar represents complex conjugation. 

We begin by sketching the method for the standard triple-well potential 
$W(z) = |z^{3} - 1|^{q}$, the minima of which are taken at the points 
$1$, $e^{\frac{2 \pi i}{3}}$, and $e^{\frac{-2 \pi i}{3}}$. We will 
construct the connection between $a=1$ and $b=e^{\frac{2\pi i}{3}}$, 
by considering the variational problem
\[ \min{ L_{p} (u)} \]
along embeddings on the plane connecting $a$ to $b$. Since the functional 
$L_{p}$ is independent of parametrizations, we choose $u: (0,1) \to \mathbb{R}^{2}$ with
$u(0) = a$, $u(1) = b$, and set $z (t) = u^{1} (t) + u^{2} (t)$. Then,
\[ L_{p} (u) = \int_{0}^{1} |z ' (\tau)| |z^{3} (\tau) - 1| \, d\tau = \int_{0}^{1} \left| \frac{d}{d\tau} g (z (\tau)) \right| d\tau = \int_{0}^{1} |w' (\tau)| \, d\tau, \]
where $w = g(z) = z - z^{4}/4$. It is clear that minimizing $L_{p}$ over the 
set of curves connecting $a$ to $b$  is reduced to the simple problem of 
minimizing the length functional on the $w$-plane for curves connecting 
$g(a) = 3/4$ to $g (b) = 3 e^{\frac{2 \pi i}{3}}/4$, which of course is 
minimized by the line segment connecting these image points. Now, 
by choosing the following parametrization for the line segment 
\[ g (z (\tau)) = \tau g(a) + (1- \tau) g(b) = \frac{3}{4} ( \tau + (1- \tau) e^{\frac{2 \pi i}{3}} ), \text{ for } 0 \leq \tau \leq1, \]
we can show that the curve $z (\tau) = r (\tau) e^{i \theta (\tau)}$ 
satisfies the parameter-free equation
\[ 4 r \cos{ (\theta - \frac{\pi}{3} ) } = r^{4} \cos{ ( 4 \theta - \frac{\pi}{3} )} + 3 \cos{ \frac{\pi}{3}}, \text{ for } 0 \leq \theta \leq \frac{\pi}{3} \text{ and } 0<r<1, \]
which is exactly the same equation presented in \cite{ABC} for $p=q=2$. 
The dependence on $p$ is through the parametrization
\[ \frac{dt}{dx} = \frac{\sqrt[p]{W(z(t))}}{|z'(t)|}, \text{ with } t(0)=\frac{1}{2}, \]
which leads to the connection $u(x) = z(t(x))$. So, naturally, one would 
expect that the theory presented in \cite{ABC} could be extended for 
any $p>1$, and although this is true, it is not without some extra technical effort.

\begin{theorem}
\label{thmthree}
Let $W(u_{1},u_{2}) = |f(z)|^{q}$, where $f = g'$ is holomorphic in an 
open subset $D$ of $\mathbb{R}^{2}$, and let the point $(u_{1}, u_{2})$ 
be identified with the complex number $z(x) = u_{1} (x) + i u_{2} (x)$. 
Additionally, let $\gamma = \{ u(x): x \in (a,b) \}$ be a smooth 
curve in $D$ and $x$ an equipartition parameter, that is, $|u_{x}|^{p} = W(u)$. 
Also, set $\alpha = u(a)$ and $\beta = u(b)$. Then, $u$ is a solution to  
\begin{equation}
\label{three}
(|u_{x}|^{p-2} u_{x})_{x} - \frac{1}{q} \nabla W (u) = 0, \text{ in } (a,b)
\end{equation}
if and only if
\begin{equation}
\label{four}
\text{\em Im}{ \left( \frac{g(z) - g(\alpha)}{g(\beta) - g(\alpha)} \right)} = 0, \text{ for all } z \in \gamma.
\end{equation}
\end{theorem}

\begin{proof}
Let $u = (u_{1}, u_{2}): (a,b) \to D$ be a solution to \eqref{three} with $|u_{x}|^{p} = W(u)$. Also,
let $L$ be the total arclength and $l$ the arclength parameter defined by 
\begin{equation*}
L = \int_{a}^{b} |u_{x}| \sqrt[q]{W(u)} \, dx \text{ and }l = \int_{a}^{x} |u_{x} (y) | \sqrt[q]{W(u(y))} \, dy.
\end{equation*}
We will show that $g(\gamma)$ is the line segment $[ g (\alpha), g (\beta) ]$. 
We begin by modifying equation \eqref{three}. Here we note that the equipartition relationship gives
\begin{equation}
\label{five}
|z_{x}|^{p} = W(u) = (f \overline{f})^{\frac{q}{2}}
\end{equation}
and that since
\begin{align*}
& \frac{\partial W}{\partial u_{1}} = (|f(z)|^{q})_{u_{1}} = [ ( f \overline{f})^{\frac{q}{2}} ]_{u_{1}} = \frac{q}{2} (f \overline{f})^{\frac{q-2}{2}} (f' \overline{f} + f \overline{f'}),\\[4pt]
& \frac{\partial W}{\partial u_{2}} = (|f(z)|^{q})_{u_{2}} = \frac{q}{2} (f \overline{f})^{\frac{q-2}{2}} (i f' \overline{f} - i f \overline{f'}),
\end{align*}
we have
\begin{equation}
\label{six}
\frac{\partial W}{\partial u_{1}} + i \frac{\partial W}{\partial u_{2}} = \frac{q}{2} (f \overline{f})^{\frac{q-2}{2}} (f' \overline{f} + f \overline{f'} - f' \overline{f} + f \overline{f'}) = q (f \overline{f})^{\frac{q-2}{2}} f \overline{f'}.
\end{equation}
Hence, based on \eqref{six}, equation \eqref{three} can be written as
\begin{align*}
0 & = (|u_{x}|^{p-2} u_{x})_{x} - \frac{1}{q} \nabla W (u)
\\
& = \frac{p}{2} |z_{x}|^{p-2} z_{xx} + (\frac{p}{2} - 1) |z_{x}|^{p-4} z_{x}^{2} \overline{z_{xx}} - (f \overline{f})^{\frac{q-2}{2}} f \overline{f'}.
\end{align*}
Now, by multiplying the above equation by $| z_{x} |^{4 - p}$ and by further simplifying, 
we see that equation \eqref{three} is equivalent to 
\begin{equation}
\label{seven}
p |z_{x}|^{2} z_{xx} + (p-2) z_{x}^{2} \overline{z_{xx}} - 2 |z_{x}|^{2(3-p)} f \overline{f'} = 0. 
\end{equation}
In addition, by differentiating the equipartition relationship $|z_{x}|^{p} = (f \overline{f})^{\frac{q}{2}}$, we obtain the equation
\[
p |z_{x}|^{p-2} (z_{xx} \overline{z_{x}} + z_{x} \overline{z_{xx}}) = q (f \overline{f})^{\frac{q-2}{2}} (f' \overline{f} z_{x} + f \overline{f'} \overline{z_{x}}), \nonumber
\]
which simplified, is equivalent to
\begin{equation}
\label{eight} 
(p-1) |z_{x}|^{2(p-2)} (z_{xx} \overline{z_{x}} + z_{x} \overline{z_{xx}}) = f' \overline{f} z_{x} + f \overline{f'} \overline{z_{x}}.
\end{equation}
Finally, differentiating the function $g$, we have
\[
\frac{dg(z)}{dl} = \frac{g'(z) z_{x}}{(f \overline{f})^{\frac{q}{2}}} = \frac{f z_{x}}{(f \overline{f})^{\frac{q}{2}}} = \frac{z_{x}}{f^{\frac{q}{2} - 1} \overline{f}^{\frac{q}{2}}} = \frac{z_{x}}{f^{\frac{q-2}{2}} \overline{f} ^{\frac{q}{2}}}
\]
and
\begin{align}
\frac{d^{2}g(z)}{dl^{2}} & = \frac{f \overline{f} z_{xx} + \frac{2-q}{2} \overline{f} f ' z_{x}^{2} - \frac{q}{2} |z_{x}|^{2} f \overline{f'}}{f^{q} \overline{f}^{q+1}} \nonumber
\\
& =\frac{f \overline{f} z_{xx} + \frac{p-2}{2(p-1)} \overline{f} f ' z_{x}^{2} - \frac{p}{2(p-1)} |z_{x}|^{2} f \overline{f'}}{f^{q} \overline{f}^{q+1}} \nonumber
\\
& =\frac{2 (p-1) f \overline{f} z_{xx} + (p-2) \overline{f} f ' z_{x}^{2} - p |z_{x}|^{2} f \overline{f'}}{2(p-1)f^{q} \overline{f}^{q+1}} \nonumber
\\
& =\frac{2 (p-1) f \overline{f} z_{xx} + z_{x} ( (p-2) \overline{f} f ' z_{x} - p \overline{z_{x}} f \overline{f'} ) }{2 (p-1) f^{q} \overline{f}^{q+1}} \nonumber
\\
\label{nine}
& =\frac{2 (p-1) f \overline{f} z_{xx} + z_{x} ( (p-2) ( \overline{f} f ' z_{x} + \overline{z_{x}} f \overline{f'} ) - 2 (p-1) \overline{z_{x}} f \overline{f'} ) }{2 (p-1) f^{q} \overline{f}^{q+1}}.
\end{align}
Now utilizing \eqref{eight}, equation \eqref{nine} becomes 
\begin{align}
\label{ten}
\frac{d^{2}g(z)}{dl^{2}} & =  \frac{2 f \overline{f} z_{xx} + (p-2) |z_{x}|^{2} |z_{x}|^{2p-4} z_{xx}}{2 f^{q} \overline{f}^{q+1}} \nonumber \\
& \quad + \frac{(p-2) |z_{x}|^{2 (p-2)} z_{x}^{2} \overline{z_{xx}} - 2 |z_{x}|^{2} f \overline{f'}}{2 f^{q} \overline{f}^{q+1}} 
\end{align}
and by virtue of \eqref{five}, equation \eqref{ten} becomes 
\begin{align*}
\frac{d^{2}g(z)}{dl^{2}} & = \frac{2 |z_{x}|^{2 (p-1)} z_{xx} + (p-2) |z_{x}|^{2 (p-1)} z_{xx}}{2 f^{q} \overline{f}^{q+1}} \\
& \quad + \frac{(p-2) |z_{x}|^{2 (p-2)} z_{x}^{2} \overline{z_{xx}} - 2 |z_{x}|^{2} f \overline{f'}}{2 f^{q} \overline{f}^{q+1}},
\end{align*}
which after simplification can be written as
\begin{align*}
\frac{d^{2}g(z)}{dl^{2}} & = \frac{p |z_{x}|^{2 (p-1)} z_{xx} + (p-2) |z_{x}|^{2 (p-2)} z_{x}^{2} \overline{z_{xx}} - 2 |z_{x}|^{2} f \overline{f'}}{2 f^{q} \overline{f}^{q+1}}
\\
& = \frac{p |z_{x}|^{2} z_{xx} + (p-2) z_{x}^{2} \overline{z_{xx}} - 2 |z_{x}|^{2 (3-p)} f \overline{f'}}{2 f^{q} \overline{f}^{q+1} |z_{x}|^{2 (2-p)}} \\
& = 0,
\end{align*}
as a result of \eqref{seven}. Thus $dg(z) / dl = C$ is constant. 
Integrating this equation and evaluating it at $l = L$ gives 
respectively $g(z) = g(\alpha) + C l$ and $C l = g(\beta) - g(\alpha)$. In addition, we have 
\[
\left| \frac{dg(z)}{dl} \right| = \left| \frac{z_{x}}{f^{\frac{q-2}{2}} \overline{f}^{\frac{q}{2}}} \right| = \frac{ |z_{x}| |z_{x}|^{\frac{p}{q}}}{ |z_{x}|^{p}} = \frac{|z_{x}| |z_{x}|^{p-1}}{|z_{x}|^{p}} = 1 = |C|,
\]
hence $L = |g (\beta) - g (\alpha)|$ and $C = \frac{g(\beta) - g(\alpha)}{|g(\beta) - g(\alpha)|}$,
which means that 
\[
g (z (l)) = \frac{L-l}{L} g(\alpha) + \frac{l}{L} g(\beta).
\]
Thus, $g$ is the desired line segment.

For the converse, assume that $\gamma = u ( (a,b) )$ satisfies \eqref{four} 
and the parameter $x$ is an equipartition parameter for $u$. Then, equation \eqref{four} can be written as 
\[
g (z) - g (\alpha) = s (x) ( g (\beta) - g (\alpha) ),
\]
where $s (x)$ is a real-valued function. Upon differentiation, we obtain 
\[
s_{x} ( g (\beta) - g (\alpha) ) = g' (z) z_{x} = f (z) z_{x}.
\]
This equation implies that 
\[
|s_{x}| = \frac{|f(z)| |z_{x}|}{ |g (\beta) - g (\alpha)|} = \frac{|f(z)|^{q}}{|g (\beta) - g (\alpha)|} > 0
\]
and since $s (x) \in \mathbb{R}$, with $s (a) = 0$ and $s (b) = 1$, we must have $s_{x} (x) > 0$. Hence,
\[
s_{x} (x) = \frac{|f(z)|^{q}}{|g (\beta) - g (\alpha)|}.
\]
Consequently, 
\begin{equation}
\label{eleven}
z_{x} = \frac{s_{x} ( g (\beta) - g (\alpha) ) }{f(z)} = C \frac{(f \overline{f})^{\frac{q}{2}}}{f},
\end{equation}
where $C = \frac{g (\beta) - g (\alpha)}{|g (\beta) - g (\alpha)|}$.

Lastly, utilizing \eqref{eleven}, we construct the differential equation \eqref{seven} as follows. First, we have 
\begin{align*}
& |z_{x}|^{2} = (f \overline{f})^{q-1},  
\\
& z_{xx} = \frac{q-2}{2} C^{2} f^{q-3} \overline{f}^{q} f ' + \frac{q}{2} f^{q-1} \overline{f}^{q-2} \overline{f'},
\\
& z_{x}^{2} = C^{2} f^{q-2} \overline{f}^{q},
\\
& \overline{z_{xx}} = \frac{q-2}{2} \overline{C}^{2} \overline{f}^{q-3} f^{q} \overline{f '} + \frac{q}{2} \overline{f}^{q-1} f^{q-2} f ',
\end{align*}
so from the above relations it follows that
\begin{align}
\label{twelve}
p |z_{x}|^{2} z_{xx} & = \frac{p q - 2 p}{2} C^{2} f^{2 q - 4} \overline{f}^{2 q - 1} f ' + \frac{p q}{2} f^{2 q - 2} \overline{f}^{2 q - 3} \overline{f '}
\\
\label{thirteen}
(p - 2) z_{x}^{2} \overline{z_{xx}} & = \frac{p q - 2 p - 2 q + 4}{2} f^{2 q - 2} \overline{f}^{2 q - 3} \overline{f '}
\\
& \quad + \frac{p q - 2 q}{2} C^{2} f^{2 q - 4} \overline{f}^{2 q - 1} f '. \nonumber
\end{align}
Adding \eqref{twelve} and \eqref{thirteen} gives
\begin{align}
\label{fourteen}
p |z_{x}|^{2} z_{xx} + (p - 2) z_{x}^{2} \overline{z_{xx}} & =  (p q - p - q ) C^{2} f^{2 q - 4} \overline{f}^{2 q - 1} f '
\\
& \quad + (p q - p - q + 2 ) f^{2 q - 2} \overline{f}^{2 q - 3} \overline{f '}, \nonumber
\end{align}
and utilizing the fact that $|z_{x}|^{2 (3 - p)} = (f \overline{f})^{2 q - 3}$, equation \eqref{fourteen} becomes
\[
p |z_{x}|^{2} z_{xx} + (p - 2) z_{x}^{2} \overline{z_{xx}} = 2 (f \overline{f})^{2 q - 3} f \overline{f '} = 2 |z_{x}|^{2 (3 - p)} f \overline{f '}.
\]
This equation is equivalent to $u$ being a solution to $q (|u_{x}|^{p - 2} u_{x})_{x} = W_{u} (u)$ and the proof is complete.
\end{proof}

The proof implies that \eqref{three} is equivalent to the first-order ordinary differential equation
\begin{equation}
\label{fifteen}
z_{x} = C \frac{(f \overline{f})^{\frac{q}{2}}}{f}, \text{ for } C \in \mathbb{C} \text{ with } |C| = 1.
\end{equation}
Multiplying \eqref{fifteen} by $\frac{f (z)}{C}$ gives
\[
\frac{d}{dx} \frac{g(z)}{C} = |f(z)|^{q} = W > 0
\]
and integrating this equation gives
\[
\text{Im} \left( \frac{g (z) - g (\alpha)}{C} \right) = 0
\]
and
\[
\frac{g (z) - g (\alpha)}{C} = \int_{a}^{x} W (z (t)) dt = l.
\]
This in particular implies that the map $x \mapsto g (z (x))$ is a one-to-one map. 
Also, in the case of $u$ being a solution, the theorem states that the set 
$g(\gamma) = \{ g(z): z \in \gamma \} $ is a line segment with end points 
$g(\alpha)$ and $g(\beta)$ and that the partial transition energy is given by 
\begin{align*}
\int_{a}^{y} \left( \frac{|u_{x}|^{p}}{p} + \frac{W(u)}{q} \right) dx & =  \int_{a}^{y} |u_{x}| \sqrt[q]{W(u)} dx \\
& =  \int_{a}^{y} \left| \frac{d}{dx} g(u) \right| dx \\ 
& =  |g (u(y)) - g (\alpha)|, \text{ for all } y \in (a,b].
\end{align*}

\begin{theorem}
\label{thmfour}
There exists at most one trajectory connecting any two minima of a holomorphic
potential, that is, if $W(z)=|f(z)|^{q}$, where $f$ is holomorphic on $\mathbb{C}$,
then there exists at most one solution of \eqref{two} that connects any two roots of $W(z)=0$.
\end{theorem}

\begin{proof}
Let $g$ be an antiderivative of $f$ and suppose that $\gamma_{1}$ and
$\gamma_{2}$ are two trajectories to \eqref{two} with the same end points $\alpha$, 
$\beta$. Since the energy $|g(\beta)-g(\alpha)|$ is positive, it follows that
$g(\beta) \neq g(\alpha)$ and we can define the function
\[ \hat{g} = \frac{|g(\beta)-g(\alpha)|}{g(\beta)-g(\alpha)}(g(z)-g(\alpha)), \text{ for all } z \in \mathbb{C}. \]
Then, $\hat{g}$ is real on $\gamma_{1} \cup \gamma_{2}$. If $\gamma_{1} \neq \gamma_{2}$,
then $\gamma_{1}$ and $\gamma_{2}$ will enclose an open domain $D$ in $\mathbb{C}$.
As the imaginary part of $\hat{g}$ on $\partial D = \gamma_{1} \cup \gamma_{2} \cup \{ \alpha, \beta \}$
is zero, it has to be identically zero in $D$. This implies that $\hat{g}$ is a constant
function in $\mathbb{C}$, which is impossible. Thus, $\gamma_{1} = \gamma_{2}$.
\end{proof}

Finally, we present some specific examples of non-existence and non-uniqueness 
of connections between the minima of various potentials. Specifically, it can be proved that for both the potentials 
\[
W(z) = |z^{n} - 1|^{q},
\]
where $n \geq 2$ is an integer, and 
\[
W(z) = |(1-z^{2})(z^{2} + \varepsilon^{2})|^{q},
\]
where $0 < \varepsilon < \infty$, there always exists a unique connection between 
each pair of their minima. We also refer to a non-existence and non-uniqueness phenomenon for the potentials 
\[
W(z) = |(1 - z^{2})(z - i \varepsilon)|^{q},
\]
where $0 \leq \varepsilon < \infty$, and 
\[
W(z) = |(z - 1)(z + a) / z|^{q},
\]
where $0 < a < 1$, respectively. In the first case it can be proved that 
there exists a connection between $-1$ and $1$ if and only if 
$|\varepsilon | > \sqrt{2 \sqrt{3} - 3}$, while in the second case there 
exist exactly two connections between $-a$ and $1$, one in the upper 
half-plane and one in the lower half-plane. We conclude by stating that 
all examples given in \cite{ABC} have exact analogs since the only 
modification needed for the transformation to the $p$-case is the 
change from potentials of the form 
$W = |f|^{2}$ (cf.\ \cite{ABC}) to potentials of the form $W = |f|^{q}$.

\section*{Acknowledgments}
The author would like to thank Professor Nicholas D. Alikakos for his valuable help and guidance.

\end{document}